\documentclass[12pt]{amsart}
\usepackage{amsthm}
\usepackage{amssymb}
\usepackage{amscd}
\usepackage[initials]{amsrefs}

\theoremstyle{plain}

\newtheorem*{thm}{Theorem}
\newtheorem*{cor}{Corollary}
\newtheorem*{prop}{Proposition}

\newcommand{\oX}{\overline{X}}
\newcommand{\ga}{\Gamma}
\newcommand{\G}{\mathbf G}

\newcommand{\Q}{{\mathbb Q}}

\newcommand{\La}{{\mathfrak a}}

\newcommand{\OO}{{\mathcal O}}
\DeclareMathOperator{\GL}{\mathbf G\mathbf L}
\newcommand{\rank}{\text{-rank}\:}

\begin{document}

\title[Topological Realization]{A Topological Realization of the Congruence Subgroup Kernel}

\author{John Scherk}
\address{Department of Computer and Mathematical Sciences\\University of
 Toronto Scarborough\\1265 Military Trail
\\Toronto, Ontario, M1C 1A4\\CANADA}
\email{scherk@math.toronto.edu}
\date{July 20, 2017}


\dedicatory{To Kumar Murty on his 60th birthday}

\maketitle

\bigskip

\noindent A number of years ago, Kumar Murty pointed out to me that the computation of the fundamental group of a Hilbert modular surface (\cite{vdG},\,IV,\S 6), and the computation of the congruence subgroup kernel of $SL(2)$ (\cite{Se}) were surprisingly similar. We puzzled over this, in particular over the role of elementary matrices in both computations. We formulated a very general result on the fundamental group of a Satake compactification of a locally symmetric space. This lead to our joint paper \cite{JMSS} with Lizhen Ji and Les Saper on these fundamental groups. Although the results in it were intriguingly similar to the corresponding calculations of the congruence subgroup kernel of the underlying algebraic group in \cite{Ra2}, we were not able to demonstrate a direct connection (cf. \cite{JMSS}, \S 7). The purpose of this note is to explain such a connection. A covering space is constructed from inverse limits of reductive Borel-Serre compactifications. The congruence subgroup kernel then appears as the group of deck transformations of this covering. The key to this is the computation of the fundamental group in \cite{JMSS}. Notations and definitions are also taken from \cite{JMSS}.

\bigskip

\noindent {\bf 1.} Let $k$ be a number field and let $S$ be a finite set of places of $k$, which
contains the infinite places $S_\infty$. Let $\G$ be a connected, absolutely almost simple, and
simply connected algebraic group defined over $k$. 

Fix a faithful representation
\begin{equation*}
\rho\colon \G \longrightarrow \GL_N
\end{equation*}
defined over $k$. Let $\OO$ the ring of \emph{$S$-integers}, and set
\begin{equation*}
\G(\OO) = \rho^{-1}(\GL_N(\OO)) \subset \G(k) .
\end{equation*}
A subgroup $\ga \subset \G(k)$ is an \emph{$S$-arithmetic subgroup} if it is commensurable with $\G(\OO)$. 

For any nonzero ideal $\La \subseteq \OO$, set
\begin{equation*}
\ga(\La) = \{\,\gamma \in \G(k) \mid  \text{$\rho(\gamma) \in \GL_N(\OO)$, 
$ \rho(\gamma) \equiv I \pmod{\La}$}\,\}\ .
\end{equation*}
A subgroup $\ga \subset \G(k)$ is an \emph{$S$-congruence subgroup} if it contains $\ga(\La)$ as a subgroup of finite index for some ideal $\La \subseteq \OO$. These two definitions are independent of the choice of $\rho$. Let ${\mathfrak M_a}$, respectively ${\mathfrak M_c}$, be the set of $S$-arithmetic subgroups, respectively $S$-congruence subgroups, of $\G$. 

We recall the definition of the congruence subgroup kernel as explained in \cite{PrasadRapinchukSurvey}. Taking each of these sets to be a fundamental system of neighbourhoods of $1$, we define two topologies, ${\mathcal T}_a$, respectively 
${\mathcal T}_c$, on $\G(k)$. Let $\widehat G(a)$, respectively $\widehat G(c)$, denote the completions of $\G(k)$ in these topologies. The corresponding completions of $\G(\OO)$ are denoted by $\widehat G(\OO,a)$, respectively $\widehat G(\OO,c)$. 

Denote by ${\mathfrak N_a}$ the set of normal subgroups of finite index of $\G(\OO)$, and by ${\mathfrak N_c}$ the set of principal $S$-congruence subgroups $\G(\OO)$. These define the topologies ${\mathcal T}_a$, respectively 
${\mathcal T}_c$, on $\G(\OO)$ as well. Then we can regard $\widehat G(\OO,a)$ and $\widehat G(\OO,c)$ as inverse limits:
\begin{equation*}
\widehat G(\OO,a) = \varprojlim_{\ga \in \mathfrak N_a} \G(\OO)/\ga
\end{equation*}
and
\begin{equation*}
\widehat G(\OO,c) = \varprojlim_{\ga \in \mathfrak N_c} \G(\OO)/\ga\ .
\end{equation*}

Since every $S$-congruence subgroup is also $S$-arithmetic, we have homomorphisms $\widehat G(a) \rightarrow \widehat G(c)$ and
$\widehat G(\OO,a) \rightarrow \widehat G(\OO,c)$. They have a common kernel, $C(S, \G)$, called the \emph {congruence subgroup kernel} . In particular, we have the exact sequence
\begin{equation}\label{exseq}
1 \rightarrow C(S,G) \rightarrow \widehat G(\OO,a) \rightarrow \widehat G(\OO,c) \rightarrow 1\ . 
\end{equation}
The {\em congruence subgroup problem} asks whether every $S$-arithmetic subgroup of $\G(k)$ is an $S$-congruence subgroup.
The congruence subgroup kernel measures to what extent this is the case.

\bigskip

\noindent {\bf 2.} The reference for the definitions in this section is \cite{JMSS},\S 4. Set $\mathbf H = \operatorname{Res}_{k/\Q} \G$,
and let $X_\infty$ be the symmetric space associated to $\mathbf H$. For each $v\in S \setminus S_\infty$,
let $X_v$ be the Bruhat-Tits building of $\G(k_v)$. Set $X = X_\infty \times \prod_{v\in S\setminus S_\infty} X_v$. Then $X$ is a
contractible, locally compact metric space. The group $G$ acts isometrically on $X$. Any $S$-arithmetic subgroup $\ga \subset \G(k)$
is a discrete subgroup of $G$  and acts properly discontinuously on $X$.

Compactifications of the quotient space $\ga\backslash X$
are of great interest. One that is particularly useful is the \emph{reductive Borel-Serre compactification}. To construct it define the reductive
Borel-Serre bordification $\overline{X}^{RBS}$ of $X$ as in \cite{JMSS}*{2.4, 4.3}. The action of a subgroup $\ga \in \mathfrak M_a$ on $X$
extends to  $\overline{X}^{RBS}$ and the quotient  $\ga\backslash\overline{X}^{RBS}$ is a compact Hausdorff space, called the
reductive Borel-Serre compactification of $\ga\backslash X$.

We can construct from $\overline{X}^{RBS}$ spaces that are analogous to the groups $\widehat G(\OO,a)$ and $\widehat G(\OO,c)$ and a sequence of maps like \eqref{exseq}. Set
\begin{equation*}
\overline{X}^{RBS}_a = \varprojlim_{\ga \in \mathfrak M_a} \ga\backslash\overline{X}^{RBS}
= \varprojlim_{\ga \in \mathfrak N_a} \ga\backslash\overline{X}^{RBS}
\end{equation*}
and
\begin{equation*}
\overline{X}^{RBS}_c = \varprojlim_{\ga \in \mathfrak M_c} \ga\backslash\overline{X}^{RBS}
= \varprojlim_{\ga \in \mathfrak N_c} \ga\backslash\overline{X}^{RBS}
\end{equation*}
They are both compact Hausdorff spaces. Let $p$ denote the natural map
\begin{equation*}
\overline{X}^{RBS}_a \xrightarrow{p} \overline{X}^{RBS}_c\ .
\end{equation*}
\medskip

\begin{prop}
\label{fundgp}
\begin{equation*}
\pi_1(\overline{X}^{RBS}_a) \cong \varprojlim_{\substack{\ga \in \mathfrak M_a \\ \ga\, \text{neat}}} \ga/E\ga \ .
\end{equation*}
\end{prop}

\begin{proof}
According to \cite{JMSS}*{Corollary~5.3}, if $\ga \in \mathfrak M_a$ and $ \ga$ is neat, then
\begin{equation*}
\pi_1 (\ga\backslash \overline{X}^{RBS}) \cong \ga/E\ga\ .
\end{equation*}
Since the groups $\ga \in \mathfrak M_a,\ \ga\ \text{neat}$ are cofinal  in $\mathfrak M_a$,
\begin{equation*}
\pi_1(\overline{X}^{RBS}_a) \cong \varprojlim_{\substack{\ga \in \mathfrak M_a \\ \ga\, \text{neat}}} \pi_1 (\ga\backslash \overline{X}^{RBS}) \ .
\end{equation*}
\end{proof}

For $\ga \in \mathfrak N_a$ there is a well-defined action on $\ga\backslash\overline{X}^{RBS}$ on the left, of $\G(\OO)$ and therefore of $ \G(\OO)/\ga$. This determines an action of $\widehat G(\OO,a)$ on $\overline{X}^{RBS}_a$. Similarly, 
$\widehat G(\OO,c)$ acts on $\overline{X}^{RBS}_c$. The map $p$ is equivariant with respect to these actions. It follows from the descriptions of $\widehat G(\OO,a)$ and $\widehat G(\OO,c)$ as inverse limits and from the definitions of  $\overline{X}^{RBS}_a$ and $\overline{X}^{RBS}_c$, that $C(S,G)$ acts transitively on each fibre of $p$. 

\bigskip

\noindent {\bf 3.} In \cite{Ra2}*{Theorem~ A, Corollary~ 1} and \cite{Margulis}*{2.4.6, I}, it is shown that  if $k\rank \G \ge 1$ and $S\rank\G \ge 2$, then for any $\ga \in \mathfrak M_a$, $E\ga \in \mathfrak M_a$. Furthermore, in this case there exists an ideal $\La \neq 0$ such that $E\ga(\La) \subset \ga$ (\cite{Ra1}*{(2.1)}). It follows that
\begin{equation*}
\widehat G(\OO,a) = \varprojlim_{\ga \in \mathfrak N_c} \G(\OO)/E\ga \ .
\end{equation*}
For any ideal $\La
\subseteq \OO$ consider the exact sequence
\begin{equation*}
1 \ \rightarrow\ \ga(\La)/E\ga(\La)\ \rightarrow
\ \G(\OO)/E\ga(\La)\ \rightarrow\ \G(\OO)/\ga(\La)
\ \rightarrow\ 1\ .
\end{equation*}
Taking projective limits over the ideals $\La$ and comparing the result with \eqref{exseq} we see that
\begin{equation} \label{csg}
C(S, G) = \varprojlim_{\ga \in \mathfrak N_c} \ga/E\ga \ .
\end{equation}
As well,
\begin{equation} \label{Xa}
\overline{X}^{RBS}_a = \varprojlim_{\ga \in \mathfrak M_c} E\ga \backslash \overline{X}^{RBS} \ .
\end{equation}

Under these rank assumptions, we obtain a simple topological realization of $C(S,G)$.
\begin{thm}
Assume that  $k\rank\, \G \ge 1$ and $S\rank\, \G \ge 2$. Then $\overline{X}^{RBS}_a$ is a simply connected covering of 
$\overline{X}^{RBS}_c$, and $C(S,G)$ acts on $\overline{X}^{RBS}_a$ as the group of deck transformations.
\end{thm}

\begin{proof}
Note that if $\ga$ is neat, then by \cite{JMSS}*{Corollary~5.2}, $\ga/E\ga$ acts freely on $E\ga \backslash \overline{X}^{RBS}$. So
\begin{equation*}
\ga(\La)/E\ga(\La)\ \rightarrow\ E\ga(\La)\backslash \overline{X}^{RBS}\ \rightarrow\ \ga(\La)\backslash \overline{X}^{RBS}
\end{equation*}
is a map of covering spaces, with $\ga(\La)/E\ga(\La)$ the group of deck transformations.
Now \eqref{csg} and \eqref{Xa} imply that $C(S, G)$ acts freely on $\overline{X}^{RBS}_a$. Taking projective limits as above, it follows that 
\begin{equation*}
C(S, G)\ \rightarrow\ \overline{X}^{RBS}_a\ \xrightarrow{p}\ \overline{X}^{RBS}_c
\end{equation*}
is a map of covering spaces, with $C(S, G)$ the group of deck transformations.

Since the groups $E\ga \in \mathfrak M_a,\ \ga\ \text{neat}$, are cofinal under the rank assumption,  the Proposition implies that \ref{fundgp} $\overline{X}^{RBS}_a$ is simply connected.
\end{proof}
\medskip

\begin{cor}
\begin{equation*}
C(S,\G) \cong  \pi_1(\overline{X}^{RBS}_c) \cong \varprojlim\limits_{\ga \in \mathfrak N_c} \pi_{1}(\ga\backslash\oX^{RBS})\ .
\end{equation*}
\end{cor}
\noindent Compare \cite{JMSS}*{\S7 (12)}

\end{document}